\documentclass[12pt]{amsart}
\usepackage{amsmath,amssymb,amsbsy,amsfonts,latexsym,amsopn,amstext,
                                               amsxtra,euscript,amscd,bm}
                    
\usepackage{url}
\usepackage[colorlinks,linkcolor=blue,anchorcolor=blue,citecolor=blue]{hyperref}
\usepackage{color}

\begin{document}
\newtheorem{problem}{Problem}
\newtheorem{theorem}{Theorem}
\newtheorem{lemma}[theorem]{Lemma}
\newtheorem{claim}[theorem]{Claim}
\newtheorem{cor}[theorem]{Corollary}
\newtheorem{prop}[theorem]{Proposition}
\newtheorem{definition}{Definition}
\newtheorem{question}[theorem]{Question}

\def\cA{{\mathcal A}}
\def\cB{{\mathcal B}}
\def\cC{{\mathcal C}}
\def\cD{{\mathcal D}}
\def\cE{{\mathcal E}}
\def\cF{{\mathcal F}}
\def\cG{{\mathcal G}}
\def\cH{{\mathcal H}}
\def\cI{{\mathcal I}}
\def\cJ{{\mathcal J}}
\def\cK{{\mathcal K}}
\def\cL{{\mathcal L}}
\def\cM{{\mathcal M}}
\def\cN{{\mathcal N}}
\def\cO{{\mathcal O}}
\def\cP{{\mathcal P}}
\def\cQ{{\mathcal Q}}
\def\cR{{\mathcal R}}
\def\cS{{\mathcal S}}
\def\cT{{\mathcal T}}
\def\cU{{\mathcal U}}
\def\cV{{\mathcal V}}
\def\cW{{\mathcal W}}
\def\cX{{\mathcal X}}
\def\cY{{\mathcal Y}}
\def\cZ{{\mathcal Z}}

\def\A{{\mathbb A}}
\def\B{{\mathbb B}}
\def\C{{\mathbb C}}
\def\D{{\mathbb D}}
\def\E{{\mathbb E}}
\def\F{{\mathbb F}}
\def\G{{\mathbb G}}
\def\I{{\mathbb I}}
\def\J{{\mathbb J}}
\def\K{{\mathbb K}}
\def\L{{\mathbb L}}
\def\M{{\mathbb M}}
\def\N{{\mathbb N}}
\def\O{{\mathbb O}}
\def\P{{\mathbb P}}
\def\Q{{\mathbb Q}}
\def\R{{\mathbb R}}
\def\S{{\mathbb S}}
\def\T{{\mathbb T}}
\def\U{{\mathbb U}}
\def\V{{\mathbb V}}
\def\W{{\mathbb W}}
\def\X{{\mathbb X}}
\def\Y{{\mathbb Y}}
\def\Z{{\mathbb Z}}

\def\ep{{\mathbf{e}}_p}
\def\em{{\mathbf{e}}_m}
\def\eq{{\mathbf{e}}_q}

\def\scr{\scriptstyle}
\def\\{\cr}
\def\({\left(}
\def\){\right)}
\def\[{\left[}
\def\]{\right]}
\def\<{\langle}
\def\>{\rangle}
\def\fl#1{\left\lfloor#1\right\rfloor}
\def\rf#1{\left\lceil#1\right\rceil}
\def\le{\leqslant}
\def\ge{\geqslant}
\def\eps{\varepsilon}
\def\mand{\qquad\mbox{and}\qquad}

\def\sssum{\mathop{\sum\ \sum\ \sum}}
\def\ssum{\mathop{\sum\, \sum}}
\def\ssumw{\mathop{\sum\qquad \sum}}

\def\vec#1{\mathbf{#1}}
\def\inv#1{\overline{#1}}
\def\num#1{\mathrm{num}(#1)}
\def\dist{\mathrm{dist}}

\def\fA{{\mathfrak A}}
\def\fB{{\mathfrak B}}
\def\fC{{\mathfrak C}}
\def\fU{{\mathfrak U}}
\def\fV{{\mathfrak V}}

\newcommand{\bflambda}{{\boldsymbol{\lambda}}}
\newcommand{\bfxi}{{\boldsymbol{\xi}}}
\newcommand{\bfrho}{{\boldsymbol{\rho}}}
\newcommand{\bfnu}{{\boldsymbol{\nu}}}

\def\GL{\mathrm{GL}}
\def\SL{\mathrm{SL}}

\def\Hba{\overline{\cH}_{a,m}}
\def\Hta{\widetilde{\cH}_{a,m}}
\def\Hb1{\overline{\cH}_{m}}
\def\Ht1{\widetilde{\cH}_{m}}

\def\flp#1{{\left\langle#1\right\rangle}_p}
\def\flm#1{{\left\langle#1\right\rangle}_m}
\def\dmod#1#2{\left\|#1\right\|_{#2}}
\def\dmodq#1{\left\|#1\right\|_q}

\def\Zm{\Z/m\Z}

\def\Err{{\mathbf{E}}}

\newcommand{\comm}[1]{\marginpar{%
\vskip-\baselineskip 
\raggedright\footnotesize
\itshape\hrule\smallskip#1\par\smallskip\hrule}}
\pagestyle{plain}
\def\xxx{\vskip5pt\hrule\vskip5pt}


\title{\bf Moments of character sums to composite modulus}

\date{\today}

 \author[B. Kerr] {Bryce Kerr}

\address{School of Mathematical Sciences, The University of New South Wales Canberra, Australia}
\email{b.kerr@adfa.edu.au}

\begin{abstract}
In this paper we consider the problem of estimating character sums to composite modulus and obtain some progress towards removing the cubefree restriction in the Burgess bound. Our approach is to estimate high order moments of character sums  in terms of solutions to congruences with Kloosterman fractions and we deal with this problem by extending some techniques of Bourgain, Garaev, Konyagin and Shparlinski and Bourgain and Garaev from the setting of prime modulus to composite modulus. As an application of our result we improve an estimate of Norton.
\end{abstract}

\maketitle


\maketitle
\section{Introduction}
Given an integer $q$ and a primitive character $\chi$ mod $q$ we  consider estimating the sums 
\begin{align}
\label{eq:charsum}
\sum_{M<n\le M+N}\chi(n).
\end{align}
The first result in this direction is  due to P\'{o}lya and Vinogradov and states that
\begin{align}
\label{eq:S111}
\sum_{M<n\le M+N}\chi(n) \ll q^{1/2}\log{q}.
\end{align}
The above bound is nontrivial provided $N\ge q^{1/2+o(1)}$ and a difficult problem is to estimate the sums~\eqref{eq:charsum} in the range $N\le q^{1/2-\delta}$. The first progress in this direction is due to Burgess~\cite{Bur0} and in a series of papers~\cite{Bur1,Bur12,Bur2,Bur3,Bur4} the work of Burgess culminated in the following estimate.
\begin{theorem}
\label{thm:Burgess}
Let $q$ be an integer and $\chi$ a primitive character mod $q$. Then we have 
\begin{align*}
\sum_{M<n\le M+N}\chi(n)\ll N^{1-1/r}q^{(r+1)/4r^2+o(1)},
\end{align*}
for any $r\le 3$ and any $r\ge 1$ provided $q$ is cubefree.
\end{theorem}
A well known conjecture states
\begin{align*}
\sum_{M<n\le M+N}\chi(n)\ll N^{1/2}q^{o(1)},
\end{align*}
and a longstanding problem is to improve on Theorem~\ref{thm:Burgess} quantitatively and in the range of parameters for which the bound is nontrivial. There has been some progress on this problem for sets of moduli with some arithmetic structure although making progress for general $q$ remains open. See~\cite{Chang,Gold,GR,Irv} for improvements to smooth modulus with origins in Heath-Brown's $q$-analogue of Weyl differencing~\cite{HB1} and~\cite{BS,Gal,Iw,Mil} for improvements to powerful modulus with the first results in this direction due to Postnikov~\cite{Pos1,Pos2}. One of the important consequences of the case $r=2$ in Theorem~\ref{thm:Burgess} is the subconvexity estimate 
\begin{align*}
L\left(\frac{1}{2},\chi \right)\ll q^{3/16+o(1)},
\end{align*}
which has recently been improved by Petrow and Young~\cite{PY} for cubefree modulus, extending earlier work of Conrey and Iwaniec~\cite{CI}.
\newline

 The restriction to cubefree modulus arises in many problems when applying the amplification method to estimate exponential sums. In the setting of Theorem~\ref{thm:Burgess} removing this restriction would allow the estimation for smaller ranges of the parameter $N$ and have applications to analytic properties of Dirichlet $L$-functions closer to the line $\Re{s}=1$. The main difficulty in achieving this  lies in the estimation of complete sums modulo prime powers. An important stage in Burgess' argument is the reduction of estimating the sums~\eqref{eq:charsum}  to the moments 
\begin{align}
\label{eq:mvde}
\sum_{\lambda=1}^{q}\left|\sum_{1\le v \le V}\chi(\lambda+v)\right|^{2r}.
\end{align}
These moments were first considered by Davenport and Erd\H{o}s~\cite{DE} for prime modulus $q$ who  
appealed to some earlier work of Davenport~\cite{Dav}. This became obsolete after Weil~\cite{Weil} whose estimates  lead to the bound 
\begin{align}
\label{eq:mvweil}
\sum_{\lambda=1}^{q}\left|\sum_{1\le v \le V}\chi(\lambda+v)\right|^{2r}\ll qV^r+q^{1/2}V^{2r}.
\end{align}
Extending the estimate~\eqref{eq:mvweil} to arbitrary composite modulus is the main obstacle in removing the cubefree restriction in Theorem~\ref{thm:Burgess}. Supposing that $q=p^{\alpha}$ is a prime power and considering~\eqref{eq:mvde}, expanding and interchanging summation gives
\begin{align}
\label{eq:mvc1}
\sum_{\lambda=1}^{p^{\alpha}}\left|\sum_{1\le v \le V}\chi(\lambda+v)\right|^{2r}\le \sum_{1\le v \le V}\left|\sum_{\lambda=1}^{p^{\alpha}}\chi\left(F_v(\lambda) \right) \right|,
\end{align}
where 
\begin{align}
\label{eq:Fvdef}
F_v(\lambda)=\frac{(v_1+\lambda)\dots (v_r+\lambda)}{(v_{r+1}+\lambda) \dots (v_{2r}+\lambda)}.
\end{align}
If $\alpha=1$  one may partition summation over $v$ into suitable sets and appeal to the Weil bound
\begin{align*}
\sum_{\lambda=1}^{p^{\alpha}}\chi\left(F_v(\lambda) \right)\ll p^{1/2},
\end{align*}
to get~\eqref{eq:mvweil}. Combining these ideas with the Chinese remainder theorem and the argument of Burgess gives Theorem~\ref{thm:Burgess}  for squarefree modulus. When $\alpha>1$, considering the sums 
\begin{align}
\label{eq:Fcomp}
\sum_{\lambda=1}^{p^{\alpha}}\chi\left(F_v(\lambda) \right),
\end{align}
one may partition $\lambda$ into residue classes mod $p^{\alpha/2}$ with the result of transforming into summation over additive characters mod $p^{\alpha/2}$, see~\cite[Chapter~12]{IwKo} for some general results related to this technique. This reduces estimating~\eqref{eq:Fcomp} to counting the number of solutions to the congruence 
\begin{align}
\label{eq:Fvpoly}
F'_v(\lambda)\equiv 0 \mod{p^{\alpha/2}}, \quad 0\le \lambda<p^{\alpha/2}.
\end{align}
If $\alpha=2$ then this is a polynomial congruence mod $p$ for which there are $O(1)$ solutions and allows the extension of~\eqref{eq:mvweil} to cubefree modulus. For arbitrary $\alpha$ we note that if $r=2$ then~\eqref{eq:Fvpoly} is a quadratic congruence whose number of solutions may be estimated via calculations with the discriminant and gives ~\eqref{eq:mvweil} for $r=2$ and arbitrary modulus. The case of $r=3$ is much more difficult and was achieved by Burgess~\cite{Bur3,Bur4} more than 20 years after the $r=2$ case. Since the work of Burgess there has been little progress on extending the estimate~\eqref{eq:mvweil} apart from some isolated values of $r$ and $\alpha$, see~\cite{Bur5,Bur6,Bur7}. These approaches are based on interpreting the average number of singular solutions to~\eqref{eq:Fvpoly}  as systems of congruences modulo divisors of $p^{\alpha/2}$ which are dealt with via a successive elimination of variables and is not clear how to generalize to larger values of $r$ and $\alpha$. In this paper we introduce an approach which allows a systematic study of the mean values~\eqref{eq:mvde} for arbitrary integers $q,r$ and in particular give the first nontrivial estimate of the moments~\eqref{eq:mvde} in the cubefull aspect for any $r\ge 4$.
\newline

 Our first step is to take advantage of summation over $v$ to reduce estimating~\eqref{eq:Fvpoly} to counting solutions to congruences with Kloosterman fractions. For
integers $q,\lambda,V,r$ we let $K_{r,q}(\lambda,V)$ count the number of solutions to the congruence 
\begin{align*}
\frac{1}{\lambda+v_1}+\dots+\frac{1}{\lambda+v_r}\equiv \frac{1}{\lambda+v_{r+1}}+\dots+\frac{1}{\lambda+v_{2r}} \mod{q},
\end{align*}
with variables satifying 
\begin{align*}
|v_1|,\dots,|v_{2r}|\le V,
\end{align*}
and note the reduction to $K_{r,q}(\lambda,V)$ can be seen by using~\eqref{eq:mvc1},~\eqref{eq:Fcomp},~\eqref{eq:Fvpoly} and interchanging summation. We carry out the details of this in Section~\ref{sec:2}.  The problem of estimating $K_{r,q}(\lambda,V)$  first appears to be considered by Heath-Brown~\cite{HB} in the case $r=2, \lambda=0$ who obtained the estimate 
\begin{align}
\label{eq:HB}
K_{2,q}(0,V)\ll \left(\frac{V^{7/2}}{q^{1/2}}+V^2 \right)q^{o(1)}.
\end{align}
The case of $r\ge 3$ and $\lambda=0$ was considered by Karatsuba~\cite{Kar} who obtained sharp estimates with  restricted ranges of the parameter $V$. Bourgain and Garaev~\cite{BG,BG1} used the Geometry of numbers to remove some restrictions in Karatsuba's estimate to obtain
\begin{align}
\label{eq:BGkl}
K_{r,q}(0,V)\ll \left(\frac{V^{3r-1}}{q}+V^{r}\right)q^{o(1)}.
\end{align}
We note that both~\eqref{eq:HB} and~\eqref{eq:BGkl} fall short of the expected bound
\begin{align}
\label{eq:expected}
K_{r,q}(\lambda,V)\ll \left(\frac{V^{2r}}{q}+V^{r}\right)q^{o(1)}.
\end{align}
The case of arbitrary $\lambda$ is much less understood. Bourgain and Garaev~\cite{BG} have shown for $q$ prime that
\begin{align}
K_{r,q}(\lambda,V)\ll \left(\frac{V^{2r}}{q^{1/4r}}+V^{r}\right)q^{o(1)}.
\end{align}
The argument of Bourgain and Garaev does not directly apply to composite modulus and builds on a strategy of Bourgain, Garaev, Konyagin and Shparlinski~\cite{BGKS} who in a series of papers~\cite{BGKS0,BGKS,BGKS1} obtain some estimates and applications  for counting the number of solutions to the congruence 
\begin{align}
\label{eq:multcong}
(v_1+\lambda)\dots (v_r+\lambda)\equiv (v_{r+1}+\lambda)\dots (v_{2r}+\lambda) \mod{q},
\end{align}
with variables satisfying 
\begin{align}
\label{eq:Vvarvar}
1\le v_1,\dots, v_{2r}\le V.
\end{align}
 We give a brief overview of the strategy of Bourgain, Garaev, Konyagin and Shparlinski~\cite{BGKS} and indicate the ideas required to extend from prime to arbitrary modulus, the details of which are given in Section~\ref{sec:3}.
\newline

 Considering solutions to the congruence~\eqref{eq:multcong}, after removing diagonal terms we are left to consider solutions such that the polynomial
\begin{align*}
P_v(X)=\prod_{i=1}^{r}(X+v_i)-\prod_{i=r+1}^{2r}(X+v_i),
\end{align*}
is not constant. Since $P_v(\lambda)\equiv 0 \mod{q},$ each solution gives us a point of the lattice 
\begin{align*}
\cL=\{ (x_0,\dots,x_{2r-1})\in \Z^{2r-1} \ : \ x_0+x_1\lambda+\dots+x_{2r-1}\lambda^{2r-1}\equiv 0 \mod{q}\},
\end{align*} 
and hence a large number of solutions allows us to construct a small lattice point. From this we obtain a polynomial $Q$ with small coefficients and $Q(\lambda)\equiv 0 \mod{q}$. Since $q$ is prime and each $P_v$ and $Q$ have a common root over $\F_q$, their resultant must vanish
\begin{align*}
\text{Res}(Q,P_v)\equiv 0 \mod{q}.
\end{align*}
If $V$ is sufficiently small then $\text{Res}(Q,P_v)<q$ and hence 
\begin{align*}
\text{Res}(Q,P_v)=0.
\end{align*}
This implies that for some root $\sigma$ of $Q$
$$(v_1+\sigma)\dots (v_r+\sigma)= (v_{r+1}+\sigma)\dots (v_{2r}+\sigma),$$
and reduces the problem to counting divisors in some ring of algebraic integers. The same strategy was applied by Bourgain and Garaev~\cite{BG} to $K_{r,q}(\lambda,V)$ who required an estimate for the number of solutions to the equation 
\begin{align*}
\frac{1}{v_1+\sigma}+\dots+\frac{1}{v_r+\sigma}=\frac{1}{v_{r+1}+\sigma}+\dots+\frac{1}{v_{2r}+\sigma},
\end{align*}
and were able to detect square root cancellation, see~\cite[Lemma~6]{BG}. The main obstacle in extending this argument to composite modulus is the fact that over a field the resultant of two polynomials vanishes if and only if they have a common root and may not be true for residue rings. We get around this issue by showing some calculations with the resultant also hold for residue rings provided our root is coprime to the modulus then use the fact that any short interval $\cI$ contains an integer coprime to $q$. This allows for a reduction of estimating $K_{r,q}(\lambda,V)$ to the case $(\lambda,q)=1$.
\newline

The main obstacle preventing further progress through this method is obtaining a sharp bound for $K_{r,q}(\lambda,V)$ uniformly over $q,\lambda,V$ and note the conjectured estimate~\eqref{eq:expected} implies~\eqref{eq:mvweil} for any integer $r$ provided $q$ is a prime power. For the case of arbitrary $q$ one would need an estimate of the strength~\eqref{eq:expected} when the variables run through intervals of differing side length owing to the use of the Chinese remainder theorem which interferes with lengths of summation when performing the reduction to $K_{r,q}(\lambda,V)$. 
One may always apply H\"{o}lder's inequality to reduce to equal side lengths although this is not sufficient for applications to a sharp bound as it loses information about domination of terms $V^{2r}/q$ and $V^{r}$.

\subsection*{Acknowledgement:} The author would like to thank Igor Shparlinski and Tim Trudgian for useful comments.
\section{Main results}
\begin{theorem}
\label{thm:main1}
Let $q$ be an integer with decomposition
$$q=q_1sc,$$
with $q_1$ squarefree, $s$ a square with $s^{1/2}$ squarefree and $c$ cubefull. For any primitive character $\chi$ mod $q$ and integer $V$ we have 
\begin{align*}
&\sum_{\lambda=1}^{q}\left|\sum_{1\le v \le V}\chi(\lambda+v) \right|^{2r}\ll qV^{r}
\\&+q^{1/2+o(1)}s^{1/4}c^{1/2}V^{r+1/2}+q^{1/2+o(1)}c^{1/2-1/16r(r-1)}V^{2r}.
\end{align*}
\end{theorem}
The estimate of Theorem~\ref{thm:main1} may be stated in the following less precise form.
\begin{cor}
\label{cor:main1}
Let $q$ be an integer with cubefull part $c$. For any primitive character $\chi$ mod $q$ and integer $V$ we have 
\begin{align*}
&\sum_{\lambda=1}^{q}\left|\sum_{1\le v \le V}\chi(\lambda+v) \right|^{2r}\ll qV^{r}
\\&+q^{3/4}c^{1/4}V^{r+1/2}q^{o(1)}+q^{1/2}c^{1/2-1/16r(r-1)}V^{2r}q^{o(1)}.
\end{align*}
\end{cor}
In applications one usually takes $V\sim q^{1/2r}$ and in this range the term $q^{3/4}c^{1/2}V^{r+1/2}$ can be ignored, provided $c$ is suitably small. Corollary~\ref{cor:main1} should be compared with the estimate 
\begin{align*}
&\sum_{\lambda=1}^{q}\left|\sum_{1\le v \le V}\chi(\lambda+v) \right|^{2r}\ll qV^{r}+q^{1/2}c^{1/2}V^{2r}q^{o(1)},
\end{align*}
obtained from the argument of Burgess and treating summation over cubefull terms trivially.
\newline

 Using Corollary~\ref{cor:main1} and well known techniques we deduce the following character sum estimate.
\begin{theorem}
\label{thm:charsumest}
Let $q$ be an integer with cubefull part $c$. For any primitive character $\chi$ mod $q$ and integers $M,N$ 
we have 
\begin{align*}
\sum_{M\le n \le M+N}\chi(n)\ll N^{1-1/r}q^{(r+1)/4r^2+o(1)}c^{(r-1)/4r^2-1/32r^3}.
\end{align*}
\end{theorem}
Comparing the estimate of Theorem~\ref{thm:charsumest} with previous results, we note that Norton~\cite[Theorem~1.6]{Nor} has obtained
\begin{align*}
\sum_{M\le n \le M+N}\chi(n)\ll c^{3/4r}N^{1-1/r}q^{(r+1)/4r^2+o(1)},
\end{align*}
and hence our bound is sharper in the $c$ aspect. We note that the estimate of Norton also contains a factor involving the order of $\chi$ although this is redundant in our setting from the assumption $\chi$ is primitive. 
\section{Reduction to equations with Kloosterman fractions}
\label{sec:2}
The main result of this section is a  reduction  of mean values of character sums to counting solutions to  congruences with Kloosterman fractions. Given integers $q,\lambda,V,r$ we recall that $K_{r,q}(\lambda,V)$ counts the number of solutions to the congruence 
\begin{align}
\label{eq:Kdefeqn}
\frac{1}{\lambda+v_1}+\dots+\frac{1}{\lambda+v_r}\equiv \frac{1}{\lambda+v_{r+1}}+\dots+\frac{1}{\lambda+v_{2r}} \mod{q},
\end{align}
with variables satifying 
\begin{align}
\label{eq:Kdefcond}
|v_1|,\dots,|v_{2r}|\le V.
\end{align}

\begin{lemma}
\label{lem:charKloos}
Let $q$ be an integer with factorization 
$$q=q_1\prod_{k\in \cK}p_k^2\prod_{i\in \cI}p_i^{2\alpha_i}\prod_{j\in \cJ}p_{j}^{2\beta_j+1},$$
with $q_1$ squarefree, $\cK,\cI,\cJ$ disjoint sets of integers and $\alpha_i\ge 2, \beta_j\ge 1$, and define $q_2,\dots,q_5$ by

\begin{align}
\label{eq:qsdef}
q_2=\prod_{k\in \cK}p_k, \quad q_3=\prod_{i\in \cI}p_i^{\alpha_i}, \quad  q_4=\prod_{j\in \cJ}p_j^{\beta_j}, \quad q_5=\prod_{j\in \cJ}p_j,
\end{align}
where we let $p_i$ denote the $i$-th prime. For any primitive character $\chi$ mod $q$ and integer $V\le q$ we have 
\begin{align*}
&\sum_{\lambda=1}^{q}\left|\sum_{1\le v \le V}\chi(\lambda+v) \right|^{2r}\ll qV^{r}+q^{1/2+o(1)}q_3q_4V \\ & \times \sum_{\substack{d|q_5 \\\ t_2\dots t_{2r}|q_1 \\ s_2\dots s_{2r}|q_2 \\ t_js_j \ll V}}(t_2\dots t_{2r})^{1/2}d^{1/2}s_2\dots s_{2r}\prod_{j=2}^{2r}\max_{\lambda}K_{r,q_3q_4d}(\lambda,V/t_js_j)^{1/2r}.
\end{align*}

\end{lemma}

We adopt the following notation throughout this section. Given a $2r$-tuple of integers $v=(v_1,\dots,v_{2r})$ we define the polynomials
\begin{align}
\label{eq:ffdef}
f_{v_1}(x)=\prod_{j=1}^{r}(x-v_j), \quad f_{v_2}(x)=\prod_{j=1}^{r}(x-v_{j+r}).
\end{align}
Given an integer $q$ we let $N_{v}(q)$ count the number of solutions to the congruence 
\begin{align}
\label{eq:Nvequation}
f'_{v_1}(\lambda)f_{v_2}(\lambda)-f_{v_1}(\lambda)f'_{v_2}(\lambda)\equiv 0 \mod{q}, 
\end{align}
with variable $\lambda$ satisfying
\begin{align}
\label{eq:Nvcondition}
(f_{v_1}(\lambda)f_{v_2}(\lambda),q)=1, \quad  0\le \lambda < q.
\end{align}
The following is a direct application of the Chinese remainder theorem.
\begin{lemma}
\label{lem:Nmult}
For $q_1$ and $q_2$ coprime we have 
\begin{align*}
N_v(q_1)N_v(q_2)\le N_v(q_1q_2).
\end{align*}
\end{lemma}

We recall some results of Burgess. The following is~\cite[Lemma~2]{Bur1}.
\begin{lemma}
\label{lem:evenprimepower}
Let $p$ be prime, $\alpha$ an integer and $\chi$ a primitive character mod $p^{2\alpha}$. We have 
\begin{align*}
\left|\sum_{\lambda=1}^{p^{2\alpha}}\chi(f_{v_1}(\lambda))\overline \chi(f_{v_2}(\lambda))\right|\le p^{\alpha}N_{v}(p^{\alpha}).
\end{align*}
\end{lemma}
The following is~\cite[Lemma~3]{Bur1}.
\begin{lemma}
\label{lem:oddevenprime}
Let $\alpha$ be an integer and $\chi$ a primitive character mod $2^{2\alpha+1}$.  We have 
\begin{align*}
\left|\sum_{\lambda=1}^{2^{2\alpha+1}}\chi(f_{v_1}(\lambda))\overline \chi(f_{v_2}(\lambda))\right|\le 2^{\alpha+1}N_{v}(2^{\alpha}).
\end{align*}
\end{lemma}
The following is ~\cite[Lemma~4]{Bur1}.
\begin{lemma}
\label{lem:oddprimepower}
Let $p$ be prime, $\alpha\ge 1$ an integer and $\chi$ a primitive character mod $p^{2\alpha+1}$. We have 
\begin{align*}
\left|\sum_{\lambda=1}^{p^{2\alpha+1}}\chi(f_{v_1}(\lambda))\overline \chi(f_{v_2}(\lambda))\right|\le p^{\alpha+1/2}N_{v}(p^{\alpha})+p^{\alpha}N_{v}(p^{\alpha+1}).
\end{align*}
\end{lemma}

The following is~\cite[Lemma~7]{Bur1} and is based on the Weil bound and Chinese remainder theorem.
\begin{lemma}
\label{lem:squarefree}
Let $q$ be squarefree and $\chi$ a primitive character mod $q$. Let 
\begin{align*}
v=(v_{1},\dots,v_{2r}),
\end{align*}
be such that 
$$|\{v_{1},\dots,v_{2r}\}|\ge r+1.$$
For integer $j$ define
\begin{align*}
A_j(v)=\prod_{\substack{i=1 \\ i\neq j}}^{2r}(v_j-v_i).
\end{align*}
There exists some $j$ with $A_j(v)\neq 0$  such that 
\begin{align*}
\left|\sum_{\lambda=1}^{q}\chi(f_{v_1}(\lambda))\overline \chi(f_{v_2}(\lambda))\right|\le (4r)^{\tau(q)}q^{1/2}(A_j(v),q)^{1/2}.
\end{align*}
\end{lemma}
The following is~\cite[Lemma~7]{Bur2}.
\begin{lemma}
\label{lem:square}
Let $p$ be prime and suppose that $v=(v_1,\dots,v_{2r})$ satisfies $A_j(v)\neq 0$ for some $j$. Then we have 
\begin{align*}
N_v(p)\ll (A_j(v),p).
\end{align*}
\end{lemma}

\begin{lemma}
\label{lem:BurgessCRT}
Let $q$ be an integer with factorization 
$$q=q_1\prod_{k\in \cK}p_k^2\prod_{i\in \cI}p_i^{2\alpha_i}\prod_{j\in \cJ}p_{j}^{2\beta_j+1},$$
with $q_1$ squarefree, $\cK,\cI,\cJ$ disjoint sets of integers and $\alpha_i\ge 2, \beta_j\ge 1$, and define
\begin{align}
\label{eq:qsdef}
q_2=\prod_{k\in \cK}p_k, \quad q_3=\prod_{i\in \cI}p_i^{\alpha_i}, \quad  q_4=\prod_{j\in \cJ}p_j^{\beta_j}, \quad q_5=\prod_{j\in \cJ}p_j.
\end{align}
 For any primitive character $\chi$ mod $q$ and integer $V$ we have 
\begin{align*}
&\sum_{\lambda=1}^{q}\left|\sum_{1\le v \le V}\chi(\lambda+v) \right|^{2r}\ll qV^{r} \\ & \quad \quad \quad \quad \quad + q^{1/2+o(1)}\sum_{d|q_5}\frac{1}{d^{1/2}}\sum_{v\in \cV_1}(A_1(v),q_1)^{1/2}(A_1(v),q_2)N_v(q_3q_4d).
\end{align*}
\end{lemma}
\begin{proof}
Let 
$$S=\sum_{\lambda=1}^{q}\left|\sum_{1\le v \le V}\chi(\lambda+v) \right|^{2r}.$$
Expanding the $2r$-th power and interchanging summation, we have 
\begin{align*}
S\le \sum_{1\le v_1,\dots,v_{2r}\le V}\left|\sum_{\lambda=1}^{q}\chi\left(\frac{(\lambda+v_1)\dots (\lambda+v_{r})}{(\lambda+v_{r+1})\dots (\lambda+v_{2r})} \right) \right|,
\end{align*}
and with notation as above, this simplifies to 
\begin{align*}
S\le \sum_{1\le v \le V}\left|\sum_{\lambda=1}^{q}\chi(f_{v_1}(\lambda))\overline \chi (f_{v_2}(\lambda)) \right|.
\end{align*}
By the Chinese remainder theorem we may factorize 
\begin{align*}
\chi=\chi_1\prod_{k\in \cK}\chi_k \prod_{i\in \cI}\chi_i \prod_{j\in \cJ}\chi_j,
\end{align*}
where $\chi_1$ is a primitive character mod $q_1,$  $\chi_k$ is a primitive character mod $p_k^2$, $\chi_i$ is a primitive character mod $p_i^{2\alpha_i}$ and $\chi_j$ is a primitive character mod $p_j^{2\beta_j+1}.$ A second application of the Chinese remainder theorem to summation over $\lambda$ gives the  decomposition 
\begin{align*}
S\le \sum_{1\le v \le V}\sigma_1(v)\prod_{k\in \cK}\sigma_k(v)\prod_{i\in \cI}\sigma_i(v)\prod_{j\in \cJ}\sigma_j(v),
\end{align*}
where 
\begin{align*}
\sigma_1(v)&=\left|\sum_{\lambda=1}^{q_1}\chi_1(f_{v_1}(\lambda))\overline \chi_1 (f_{v_2}(\lambda)) \right|, \\
\sigma_k(v)&=\left|\sum_{\lambda=1}^{p_k^{2}}\chi_k(f_{v_1}(\lambda))\overline \chi_k (f_{v_2}(\lambda)) \right|, \\
\sigma_i(v)&=\left|\sum_{\lambda=1}^{p_i^{2\alpha_i}}\chi_i(f_{v_1}(\lambda))\overline \chi_i (f_{v_2}(\lambda)) \right|, \\
\sigma_j(v)&=\left|\sum_{\lambda=1}^{p_j^{2\beta_j+1}}\chi_j(f_{v_1}(\lambda))\overline \chi_j (f_{v_2}(\lambda)) \right|.
\end{align*}

We partition the outer summation over $v$ into sets 
\begin{align*}
\cV_{\ell}&=\{ 1\le v \le V \ : \ |\{v_1,\dots,v_{2r}\}|\ge r+1, \ \  A_{\ell}(v)\neq 0 \ \}, \quad 1\le \ell \le 2r, \\
\cV'&=\{ 1\le v \le V \ : \ |\{v_1,\dots,v_{2r}\}|\le  r\},
\end{align*}
and note $|\cV'|\ll V^r.$ Using that 
\begin{align*}
\{ (v_1,\dots,v_{2r}) \ : \ 1\le v_i\le V\} \subseteq \bigcup_{\ell=1}^{2r}\cV_{\ell}\cup \cV',
\end{align*}
and estimating terms $\sigma$ for $v\in \cV_2$ trivially gives 
\begin{align}
\label{eq:SS1}
S&\ll qV^{r}+\sum_{\ell=1}^{2r}\sum_{v\in \cV_{\ell}}\sigma_1(v)\prod_{k\in \cK}\sigma_k(v)\prod_{i\in \cI}\sigma_i(v)\prod_{j\in \cJ}\sigma_j(v)\ll qV^{r}+S_1,
\end{align}
 where 
\begin{align}
\label{eq:S1sigma}
S_1=\sum_{v\in \cV_{1}}\sigma_1(v)\prod_{k\in \cK}\sigma_k(v)\prod_{i\in \cI}\sigma_i(v)\prod_{j\in \cJ}\sigma_j(v),
\end{align}
and we have used symmetry to estimate
\begin{align*}
S_{\ell}\ll S_1.
\end{align*}
For $v\in \cV_1$, $k\in \cK, i\in \cI$ and $j\in \cJ$,  by Lemmas~\ref{lem:evenprimepower}~\ref{lem:oddevenprime},~\ref{lem:oddprimepower},~\ref{lem:squarefree} and~\ref{lem:square}
\begin{align*}
\sigma_1(v)&\ll q_1^{1/2+o(1)}(A_1(v),q_1)^{1/2}, \\
\sigma_k(v) &\ll p_kN_v(p_k) \ll p_k(A_1(v),p_k), \\
\sigma_i(v) &\ll p_i^{\alpha_i}N_v(p_i^{\alpha_i}), \\
\sigma_j(v) &\ll p_j^{\beta_j+1/2}N_v(p_j^{\beta_j})+p_j^{\beta_j}N_v(p_j^{\beta_j+1}).
\end{align*}
and hence 
\begin{align*}
&\sigma_1(v)\prod_{k\in \cK}\sigma_k(v)\prod_{i\in \cI}\sigma_i(v)\prod_{j\in \cJ}\sigma_j(v) \ll \\ 
& q^{1/2+o(1)}(A_1(v),q_1)^{1/2}(A_1(v),q_2)\prod_{i\in \cI}N_v(p_i^{\alpha_i})\prod_{j\in \cJ}\left(N_v(p_j^{\beta_j})+\frac{N_v(p_j^{\beta_j+1})}{p_j^{1/2}} \right).
\end{align*}
Recalling~\eqref{eq:qsdef} and using Lemma~\ref{lem:Nmult}, we see that 
\begin{align*}
\prod_{i\in \cI}N_v(p_i^{\alpha_i})\prod_{j\in \cJ}\left(N_v(p_j^{\beta_j})+\frac{N_v(p_j^{\beta_j+1})}{p_j^{1/2}} \right)&\le N_v(q_3)\sum_{d|q_5}\frac{N_v(q_4d)}{d^{1/2}} \\
&\le \sum_{d|q_5}\frac{N_v(q_3q_4d)}{d^{1/2}},
\end{align*}
which  implies 
\begin{align*}
&\sigma_1(v)\prod_{k\in \cK}\sigma_k(v)\prod_{i\in \cI}\sigma_i(v)\prod_{j\in \cJ}\sigma_j(v) \\
&\ll q^{1/2+o(1)}(A_1(v),q_1)^{1/2}(A_1(v),q_2)\sum_{d|q_5}\frac{N_v(q_3q_4d)}{d^{1/2}}.
\end{align*}
Substituting the above into~\eqref{eq:S1sigma} we get 
\begin{align*}
S_1\ll q^{1/2+o(1)}\sum_{d|q_5}\frac{1}{d^{1/2}}\sum_{v\in \cV_1}(A_1(v),q_1)^{1/2}(A_1(v),q_2)N_v(q_3q_4d),
\end{align*}
and the result follows from~\eqref{eq:SS1}.
\end{proof}
\section{Proof of Lemma~\ref{lem:charKloos}}
By Lemma~\ref{lem:BurgessCRT} we have 
\begin{align}
\label{eq:lemKs1}
&\sum_{\lambda=1}^{q}\left|\sum_{1\le v \le V}\chi(\lambda+v) \right|^{2r}\ll qV^{r}+ q^{1/2+o(1)}\sum_{d|q_5}\frac{1}{d^{1/2}}S_d,
\end{align}
where 
\begin{align*}
S_d=\sum_{v\in \cV_1}(A_1(v),q_1)^{1/2}(A_1(v),q_2)N_v(q_3q_4d).
\end{align*}
Fix some $d|q_5$ and consider $S_d$. Recalling that 
\begin{align*}
A_1(v)=\prod_{i\neq 1}(v_1-v_i),
\end{align*}
we partition summation over $v$ into sets depending on the values of $(A_1(v),q_1)$ and $(A_1(v),q_2)$. For $d_1|q_1$ and $d_2|q_2$ we define
\begin{align*}
\cV_1(d_1,d_2)=\{ v\in \cV_1 \ : \  (A_1(v),q_1)=d_1, \ \ (A_1(v),q_2)=d_2 \},
\end{align*}
so that 
\begin{align}
\label{eq:lemKs2}
S_d=\sum_{\substack{d_1|q_1 \\ d_2|q_2 }}d_1^{1/2}d_2S_d(d_1,d_2),
\end{align}
where 
\begin{align}
\label{eq:SdN}
S_d(d_1,d_2)=\sum_{v\in \cV_1(d_1,d_2)}N_v(q_3q_4d).
\end{align}
Since $N_v$ is defined by~\eqref{eq:Nvequation} and~\eqref{eq:Nvcondition}, we may write 
\begin{align*}
N_v(q_3q_4d)=\sum_{\substack{\lambda=0 \\ (*)}}^{q_3q_4d-1}1,
\end{align*}
where $(*)$ denotes summation with conditions 
\begin{align}
\label{eq:ff'cond}
f'_{v_1}(\lambda)f_{v_2}(\lambda)-f_{v_1}(\lambda)f'_{v_2}(\lambda)\equiv 0 \mod{q_3q_4d}, \quad (f_{v_1}(\lambda)f_{v_2}(\lambda),q_3q_4d)=1.
\end{align}
Substituting into~\eqref{eq:SdN} and rearranging summation gives 
\begin{align*}
S_d(d_1,d_2)=\sum_{\lambda=0}^{q_3q_4d-1}\sum_{\substack{v\in \cV_1(d_1,d_2) \\ (*)}}1.
\end{align*}
Recalling~\eqref{eq:ffdef}, the conditions~\eqref{eq:ff'cond}
imply that
\begin{align*}
\frac{1}{\lambda+v_1}+\dots+\frac{1}{\lambda+v_r}\equiv \frac{1}{\lambda+v_{r+1}}+\dots+\frac{1}{\lambda+v_{2r}} \mod{q_3q_4d},
\end{align*}
hence defining
\begin{align*}
K_{r,q_3q_4d}(\lambda,V,d_1,d_2),
\end{align*}
to count the number of solutions to the congruence 
\begin{align*}
\frac{1}{\lambda+v_1}+\dots+\frac{1}{\lambda+v_r}\equiv \frac{1}{\lambda+v_{r+1}}+\dots+\frac{1}{\lambda+v_{2r}} \mod{q_3q_4d},
\end{align*}
with variables satisfying 
\begin{align*}
1\le v_1,\dots,v_{2r}\le V, \quad (A_1(v),q_1)=d_1, \ \ (A_1(v),q_2)=d_2,
\end{align*}
we have 
\begin{align}
\label{eq:SK123}
S_d(d_1,d_2)\le \sum_{\lambda=0}^{q_3q_4d-1}K_{r,q_3q_4d}(\lambda,V,d_1,d_2).
\end{align}
Our next step is to estimate $K_{r,q_3q_4d}(\lambda,V,d_1,d_2)$ in terms of $K_{r,q_3q_4d}(\lambda,V).$ If $(A_1(v),q_1)=d_1$  and $(A_1(v),q_2)=d_2$, then since both $q_1$ and $q_2$ are squarefree, there exists a decomposition
\begin{align*}
d_1=t_2\dots t_{2r}, \quad d_2=s_2\dots s_{2r}, \quad (t_i,t_j)=1, \quad (s_i,s_j)=1, \ \ i\neq j,
\end{align*}
such that 
\begin{align*}
v_j\equiv v_1 \mod{t_j}, \quad v_j\equiv v_1 \mod{s_j},
\end{align*}
and since $(q_1,q_2)=1$ this implies that 
\begin{align*}
v_j\equiv v_1 \mod{t_js_j},
\end{align*}
and note that in order for $A_1(v)\neq 0$ we must have $t_js_j\ll V$.
With $s_2,\dots,t_{2r}$ as above, let  $K_{r,q_3q_4d}(\lambda,V,s_2,\dots,t_{2r})$ count the number of solutions to the congruence 
\begin{align*}
\frac{1}{\lambda+v_1}+\sum_{j=2}^{r}\frac{1}{\lambda+v_1+u_jt_js_j}\equiv \sum_{j=r+1}^{2r}\frac{1}{\lambda+v_{1}+u_jt_js_j} \mod{q_3q_4d},
\end{align*}
with variables satisfying 
\begin{align*}
1\le v_1\le V, \quad |u_j|\le \frac{V}{s_jt_j},
\end{align*}
so that 
\begin{align}
\label{eq:KKs111}
K_{r,q_3q_4d}(\lambda,V)\ll \sum_{\substack{t_2\dots t_{2r}=d_1 \\ s_2\dots s_{2r}=d_2 \\ t_js_j \ll V}}K_{r,q_3q_4d}(\lambda,V,s_1,\dots,t_{2r}).
\end{align}
Fix some $s_2,\dots t_{2r}$ and consider  $K_{r,q_3q_4d}(\lambda,V,s_2,\dots,t_{2r}).$ Estimating the contribution from $v_1$ trivially, we see that there exists some $\lambda^{*}$ such that $K_{r,q_3q_4d}(\lambda,V,d_1,d_2)$ is bounded by $O(V)$ times the number of solutions to the congruence 
\begin{align*}
\frac{1}{\lambda^{*}}+\sum_{j=2}^{r}\frac{1}{\lambda^{*}+u_jt_js_j}\equiv \sum_{j=r+1}^{2r}\frac{1}{\lambda^{*}+u_jt_js_j} \mod{q_3q_4d},
\end{align*}
with variables satisfying $|u_j|\le V/s_jt_j.$ Detecting via additive characters and using H\"{o}lder's inequality, we get 
\begin{align*}
& K_{r,q_3q_4d}(\lambda,V,s_2,\dots,t_{2r})\ll \frac{V}{q_3q_4d}\sum_{y=1}^{q_3q_4d}\prod_{j=2}^{2r}\left|\sum_{|u_j|\le V/t_js_j}e_{q_3q_4d}(y(\lambda^{*}+t_js_ju_j)^{-1}) \right| \\
& \quad \quad \quad \quad \ll V\prod_{j=2}^{2r}\left(\frac{1}{q_3q_4d}\sum_{y=1}^{q_3q_4d}\left|\sum_{|u_j|\le V/t_js_j}e_{q_3q_4d}(y(\lambda^{*}+t_js_ju_j)^{-1}) \right|^{2r}\right)^{1/2r}.
\end{align*}
Hence with $K_{r,q}(\lambda,V)$ defined as in~\eqref{eq:Kdefeqn} and~\eqref{eq:Kdefcond} we have 
\begin{align*}
K_{r,q_3q_4d}(\lambda,V,s_2,\dots,t_{2r})\ll V\prod_{j=2}^{2r}\max_{\lambda}K_{r,q_3q_4d}(\lambda,V/t_js_j)^{1/2r}.
\end{align*}
Substituting the above into~\eqref{eq:KKs111} gives
\begin{align*}
K_{r,q_3q_4d}(\lambda,V)\ll V\sum_{\substack{t_2\dots t_{2r}=d_1 \\ s_2\dots s_{2r}=d_2 \\ t_js_j \ll V}}\prod_{j=2}^{2r}\max_{\lambda}K_{r,q_3q_4d}(\lambda,V/t_js_j)^{1/2r},
\end{align*}
and hence by~\eqref{eq:SK123}  
\begin{align*}
S_d(d_1,d_2)\ll Vq_3q_4d\sum_{\substack{t_2\dots t_{2r}=d_1 \\ s_2\dots s_{2r}=d_2}}\prod_{j=2}^{2r}\max_{\lambda}K_{r,q_3q_4d}(\lambda,V/t_js_j)^{1/2r}.
\end{align*}
Combining the above with~\eqref{eq:lemKs1} and~\eqref{eq:lemKs2} gives
\begin{align*}
&\sum_{\lambda=1}^{q}\left|\sum_{1\le v \le V}\chi(\lambda+v) \right|^{2r}\ll qV^{r} \\ &+ q^{1/2+o(1)}q_3q_4V\sum_{\substack{d|q_5 \\\ d_1|q_1 \\ d_2|q_2}}d_1^{1/2}d^{1/2}d_2\sum_{\substack{t_2\dots t_{2r}=d_1 \\ s_2\dots s_{2r}=d_2 \\ t_js_j\ll V}}\prod_{j=2}^{2r}\max_{\lambda}K_{r,q_3q_4d}(\lambda,V/t_js_j)^{1/2r},
\end{align*}
and the result follows after rearranging summation.
\section{Equations with Kloosterman fractions}
\label{sec:3}
In this section we estimate $K_{r,q}(\lambda,V)$ for arbitrary integer $q$.
\begin{lemma}
\label{lem:eqnKloosterman}
Let  $K_{r,q}(\lambda,V)$ be defined by~\eqref{eq:Kdefeqn} and~\eqref{eq:Kdefcond}. For any integer $q$,  if 
$$V\ll q^{1/4k(k-1)},$$
then we have 
\begin{align*}
K_{r,q}(\lambda,V)\ll V^{r}q^{o(1)}.
\end{align*}
\end{lemma}
\begin{cor}
\label{lem:eqnKloosterman1}
Let  $K_{r,q}(\lambda,V)$ be defined by~\eqref{eq:Kdefeqn} and~\eqref{eq:Kdefcond}. For arbitrary integers $q$ and $V$ we have 
\begin{align*}
K_{r,q}(\lambda,V)\ll \left(\frac{V^{2r}}{q^{1/4(r-1)}}+V^{r}\right)q^{o(1)}.
\end{align*}
\end{cor}
We first recall some basics of linear algebra. Given an $n\times n$ matrix 
$$
A=\begin{bmatrix}
    a_{1,1} & a_{1,2} & a_{1,3} & \dots  & a_{1,n} \\
    a_{2,1} & a_{2,2} & a_{2,3} & \dots  & a_{2,n} \\
    \vdots & \vdots & \vdots & \ddots & \vdots \\
    a_{n,1} & a_{n,2} & a_{n,3} & \dots  & a_{n,n}
\end{bmatrix},
$$
let $A_{i,j}$ denote the matrix obtained by deleting the $i$-th row and $j$-th column from $A$ and define the adjoint of $A$, $\text{adj}(A)$ to be the matrix with $(i,j)$-th entry $(-1)^{i+j}\text{det}(A_{j,i}).$
Then we have 
\begin{align}
\label{eq:Aadjoint}
A\times \text{adj}(A)=\text{adj}(A)\times A=\text{det}(A)\begin{bmatrix}
    1 & 0 & 0 & \dots  & 0 \\
    0 & 1 & 0 & \dots  & 0 \\
    \vdots & \vdots & \vdots & \ddots & \vdots \\
    0 & 0 & 0 & \dots  & 1
\end{bmatrix}.
\end{align}
Given two polynomials $f,g\in \Z[X]$ with coefficients 
\begin{align}
\label{eq:fgDEF}
f(X)=a_nX^n+\dots+a_0, \quad g(X)=b_mX^m+\dots+b_0,
\end{align}
we define the Sylvester matrix $S(f,g)$ of $f$ and $g$ to be $(m+n)\times (m+n)$ matrix
\begin{align*}
S(f,g)=\begin{bmatrix}
    a_n & a_{n-1} & a_{n-2} & \dots  & 0 & 0 & 0 \\
    0 & a_n & a_{n-1} & \dots  & 0 & 0 & 0  \\
    \vdots & \vdots & \vdots & \ddots & \vdots & \vdots & \vdots \\
    0 & 0 & 0 & \dots &  a_1 & a_0  & 0 \\ 
    0 & 0 & 0 & \dots &  a_2 & a_1  & a_0 \\ 
b_m & b_{m-1} & b_{m-2} & \dots  & 0 & 0 & 0 \\
    0 & b_m & b_{m-1} & \dots  & 0 & 0 & 0  \\
    \vdots & \vdots & \vdots & \ddots & \vdots & \vdots & \vdots \\
    0 & 0 & 0 & \dots &  b_1 & b_0  & 0 \\ 
    0 & 0 & 0 & \dots &  b_2 & b_1  & b_0
\end{bmatrix},
\end{align*}
and define the resultant of $f$ and $g$ by 
\begin{align}
\label{eq:Resdef}
\text{Res}(f,g)=\text{det}(S(f,g)).
\end{align}
We recall that $\text{Res}(f,g)=0$ if and only if $f$ and $g$ have a common root over $\C$.
The following result will be needed to extend the techniques of~\cite{BG,BGKS} from prime to composite modulus.
\begin{lemma}
\label{lem:liftC}
Let $q$ and $\lambda$ be  integers with $(\lambda,q)=1$.  Suppose $f,g\in \Z[X]$ are polynomials satisfying
\begin{align}
\label{eq:FGzeromodq}
f(\lambda)\equiv g(\lambda)\equiv 0 \mod{q}.
\end{align}
Then we have 
\begin{align*}
\text{Res}(f,g)\equiv 0 \mod{q}.
\end{align*}
\end{lemma}
\begin{proof}
We may suppose $\text{Res}(f,g)\neq 0$ as otherwise the result is immediate. Let $f$ and $g$ have coefficients given by~\eqref{eq:fgDEF} and define 
$$\tilde \lambda = \begin{bmatrix}
    1  \\
    \lambda  \\
    \vdots \\
    \lambda^{m+n}
\end{bmatrix}.$$
The condition $(\lambda,q)=1$ and~\eqref{eq:FGzeromodq} imply that 
\begin{align*}
S(f,g)\tilde \lambda \equiv  \begin{bmatrix}
    0  \\
    0  \\
    \vdots \\
    0
\end{bmatrix} \mod{q},
\end{align*}
and hence by~\eqref{eq:Aadjoint} and~\eqref{eq:Resdef}
\begin{align*}
\text{Res}(f,g)\begin{bmatrix}
    1 & 0 & 0 & \dots  & 0 \\
    0 & 1 & 0 & \dots  & 0 \\
    \vdots & \vdots & \vdots & \ddots & \vdots \\
    0 & 0 & 0 & \dots  & 1
\end{bmatrix}\tilde \lambda\equiv 0 \mod{q},
\end{align*}
which implies $\text{Res}(f,g)\equiv 0 \mod{q}.$
\end{proof}
We will require the following resultant estimate of Bourgain, Garaev, Konyagin and Shparlinski~\cite[Corollary~3]{BGKS}.
\begin{lemma}
\label{lem:resultantbound}
Let $P_1(X)$ and $P_2(X)$ be nonconstant polynomials
\begin{align*}
P_1(X)=\sum_{i=0}^{M-1}a_iX^{M-1-i}, \quad P_2(X)=\sum_{i=0}^{N-1}b_iX^{N-1-i},
\end{align*}
such that 
\begin{align*}
|a_i|<H^{i+\sigma}, \quad |b_i|<H^{i+\theta}.
\end{align*}
Then we have 
\begin{align*}
\text{Res}(P_1,P_2)\ll H^{(M-1+\sigma)(N-1+\theta)-\theta\sigma}.
\end{align*}
\end{lemma}
The following is due to Bourgain and Garaev~\cite[Lemma~6]{BG}.
\begin{lemma}
\label{lem:BG}
For any fixed positive integer $r$ and all values of $\sigma \in \C$ the number of solutions to the equation
\begin{align*}
\frac{1}{\sigma+v_1}+\dots+\frac{1}{\sigma+v_r}=\frac{1}{\sigma+v_{r+1}}+\dots+\frac{1}{\sigma+v_{2r}},
\end{align*}
with variables satisfying
\begin{align*}
|x_1|,\dots,|x_{2r}|\le V,
\end{align*}
is bounded by $V^{r+o(1)}.$
\end{lemma}
The following is a well known consequence of the sieve of Eratosthenes.
\begin{lemma}
\label{lem:coprimesieve}
For any integers $M,N$ and $q$ we have 
\begin{align*}
\sum_{\substack{M<n\le M+N \\ (n,q)=1}}1=\frac{\phi(q)}{q}N+O(2^{\omega(q)}).
\end{align*}
\end{lemma}
The following is a consequence of Lemma~\ref{lem:coprimesieve} and standard estimates for arithmetic functions.
\begin{cor}
\label{cor:coprimeshort}
Let $\varepsilon>0$ be an arbitrary positive number and $q$ an integer. Then any interval $\cI$ of length $|\cI|\gg q^{\varepsilon}$ contains an integer coprime to $q$.
\end{cor}

\section{Proof of Lemma~\ref{lem:eqnKloosterman}}
Fix some sufficiently small  $\varepsilon>0$ and suppose $V\gg q^{\varepsilon}$ as otherwise the result is trivial. By Corollary~\ref{cor:coprimeshort} there exists some $\lambda^{*}$ satisfying 
\begin{align}
\label{eq:lambdastar111}
|\lambda^{*}-\lambda|\le V, \quad (\lambda^{*},q)=1.
\end{align}
If $v_1,\dots,v_{2r}$ satisfies 
\begin{align*}
\frac{1}{\lambda+v_1}+\dots+\frac{1}{\lambda+v_r}\equiv \frac{1}{\lambda+v_{r+1}}+\dots+\frac{1}{\lambda+v_{2r}} \mod{q}, \quad |v_i|\le V,
\end{align*} 
then 
\begin{align*}
\frac{1}{\lambda^{*}+u_1}+\dots+\frac{1}{\lambda^{*}+u_r}\equiv \frac{1}{\lambda^{*}+u_{r+1}}+\dots+\frac{1}{\lambda^{*}+u_{2r}} \mod{q},
\end{align*} 
where  $$u_i=v_i+(\lambda-\lambda^{*}),$$ and hence  by~\eqref{eq:lambdastar111} $|u_i|\le 2V$ which implies that
\begin{align*}
K_{r,q}(\lambda,V)\le K_{r,q}(\lambda^{*},2V).
\end{align*}
Hence it is sufficient to show that for any $\lambda$ satisfying $(\lambda,q)=1$ and integer $V$  satisfying
\begin{align}
\label{eq:Vcond}
q^{\varepsilon}\le V\ll q^{1/4r(r-1)},
\end{align}
we have 
\begin{align}
\label{eq:Kind}
K_{r,q}(\lambda,V)\ll V^{r+o(1)}.
\end{align}
We proceed by induction on $r$ and note that the case $r=1$ is trivial. We formulate our induction hypothesis as follows.  Let $k$ be an integer such that for any $r\le k-1$  the estimate~\eqref{eq:Kind} holds for any $V$ satisfying~\eqref{eq:Vcond}. Let $V$ satisfy
\begin{align}
\label{eq:Vkcond}
V\ll q^{1/4k(k-1)},
\end{align}
and we aim to show that 
\begin{align}
\label{eq:Kcontr}
K_{k,q}(\lambda,V)\ll V^{k+o(1)}.
\end{align}
Let $K'_{k,q}(\lambda,V)$ count the number of solutions to the congruence 
\begin{align}
\label{eq:Keqn44}
\frac{1}{\lambda+v_1}+\dots+\frac{1}{\lambda+v_k}\equiv \frac{1}{\lambda+v_{k+1}}+\dots+\frac{1}{\lambda+v_{2k}} \mod{q},
\end{align} 
with variables satisfying 
\begin{align}
\label{eq:K'}
|v_i|\le V, \quad |\{v_1,\dots,v_{2k}\}|=2k,
\end{align}
and let $K''_{k,q}(\lambda,V)$ count the number of solutions to the congruence~\eqref{eq:Keqn44} with variables satisfying
\begin{align}
\label{eq:K''}
|v_i|\le V, \quad |\{v_1,\dots,v_{2k}\}|<2k,
\end{align} 
so that 
\begin{align}
\label{eq:KKK'''}
K_k(\lambda,V)\le K'_k(\lambda,V)+K''_k(\lambda,V).
\end{align}
Considering $K''$, if $(v_1,\dots,v_{2k})$ satisfy~\eqref{eq:Keqn44} and~\eqref{eq:K''} then $v_i=v_j$ for some $i\neq j$ and hence 
\begin{align}
\label{eq:K''KK}
K''_k(\lambda,V)\le \sum_{1\le i<j\le 2k}K_{i,j}(\lambda,V)\ll K_{i,j}(\lambda,V),
\end{align}
for some pair $i<j$, where $K_{i,j}(\lambda,V)$ counts the number of solutions to the congruence~\eqref{eq:Keqn44} with variables satisfying~\eqref{eq:K''} and $v_i=v_j$. Fixing $v_i$ with $O(V)$ choices, we see that there exists some sequence 
$$\varepsilon_1,\dots,\varepsilon_{2k-2}\in \{-1,1\},$$
and some integer $b$ such that 
\begin{align}
\label{eq:KijK''}
K_{i,j}(\lambda,V)\ll VK'''(\lambda,V),
\end{align}
where $K'''(\lambda,V)$ counts the number of solutions to the congruence 
\begin{align*}
\frac{\varepsilon_1}{\lambda+v_1}+\dots+\frac{\varepsilon_{2k-2}}{\lambda+v_{2k-2}}\equiv b \mod{q},
\end{align*}
with variables satisfying $|v_1|,\dots,|v_{2k-2}|\le V$. Detecting via additive characters, we have 
\begin{align*}
K''(\lambda,V)=\frac{1}{q}\sum_{y=1}^{q}\prod_{j=1}^{2k-2}\left(\sum_{|v|\le V}e_q\left(y\varepsilon_j (\lambda+v)^{-1} \right) \right)e_q(-yb),
\end{align*}
and hence by H\"{o}lder's inequality 
\begin{align*}
K''(\lambda,V)\le K_{k-1}(\lambda,V).
\end{align*}
Hence by~\eqref{eq:K''KK},~\eqref{eq:KijK''} and our induction hypothesis
\begin{align*}
K''_k(\lambda,V)\ll V^{k+o(1)}.
\end{align*}
Combining with~\eqref{eq:KKK'''} it is sufficient to show that 
\begin{align}
\label{eq:K'prrr}
K'_k(\lambda,V)\ll V^{k+o(1)},
\end{align}
and hence we may suppose that $K'_k(\lambda,V)\neq 0$. For a $2k$-tuple $v=(v_1,\dots,v_{2k})$ we define the polynomial
$$P_v(X)=\prod_{i\neq 1}(X+v_i)+\dots+\prod_{i\neq r}(X+v_i)-\prod_{i\neq r+1}(X+v_i)-\dots-\prod_{i\neq 2k}(X+v_i),$$
so that $P_v$ has degree at most $2k-2$. For each $v=(v_1,\dots,v_{2r})$ satisfying~\eqref{eq:Keqn44} we have 
\begin{align*}
P_{v}(\lambda)\equiv 0 \mod{q},
\end{align*}
and the assumption that $|\{v_1,\dots,v_{2r}\}|=2r$ implies that
\begin{align*}
P_v(-v_1)\neq 0.
\end{align*}
Since 
\begin{align*}
P_v(-v_1)\ll V^{2r-1}<q,
\end{align*}
we see that $P_v(X)$ is not a constant polynomial. Writing
\begin{align*}
P_v(X)=\sum_{i=0}^{2k-2}a_iZ^{2k-2-i},
\end{align*}
  the coefficients of $P_v(X)$ satisfy
\begin{align}
 \label{eq:Pcoefficients}
|a_i|\ll V^{i+1}.
\end{align}
Fixing one point $v^{*}=(v^{*}_1,\dots,v^{*}_{2r})$ counted by $K'_k(\lambda,V)$, for any other point $v$ we have 
\begin{align*}
P_{v^{*}}(\lambda)\equiv P_{v}(\lambda)\equiv 0 \mod{q},
\end{align*}
and hence the assumption $(\lambda,q)=1$ combined with Lemma~\ref{lem:liftC} implies that 
\begin{align}
\label{eq:Res0mod}
\text{Res}(P_{v^*},P_v)\equiv 0 \mod{q}.
\end{align}
By~\eqref{eq:Pcoefficients} and Lemma~\ref{lem:resultantbound}
\begin{align*}
\text{Res}(P_{v^*},P_v)\ll V^{4k(k-1)},
\end{align*}
and hence by~\eqref{eq:Vkcond} and~\eqref{eq:Res0mod}
\begin{align*}
\text{Res}(P_{v^*},P_v)=0,
\end{align*}
so that $P_{v^{*}}$ and $P_v$ have a common root over $\C$. Let $\sigma_1,\dots,\sigma_\ell$ denote the distinct roots of $P_{v^{*}}$ over $\C$. For any $v=(v_1,\dots,v_{2k})$ counted by $K'_k(\lambda,V)$ we have 
\begin{align*}
P_v(\sigma_j)=0,
\end{align*}
for some $1\le j \le \ell$ and note the assumption that the $v_i$'s are pairwise distinct implies that $v_i\neq \sigma_j$ for any $1\le i \le 2k$. Hence defining $J(\sigma)$ to count the number of solutions to the equation 
\begin{align*}
\frac{1}{\sigma+v_1}+\dots+\frac{1}{\sigma+v_{k}}=\frac{1}{\sigma+v_{r+1}}+\dots+\frac{1}{\sigma+v_{2k}},
\end{align*}
with variables satisfying  $|v_i|\le V$ we have
\begin{align*}
K'_k(\lambda,V)\le \sum_{j=1}^{\ell}J(\sigma_j),
\end{align*}
and hence from Lemma~\ref{lem:BG}
\begin{align*}
K'_k(\lambda,V)\ll V^{k+o(1)},
\end{align*}
which establishes~\eqref{eq:K'prrr} and completes the proof.
\section{Proof of Corollary~\ref{lem:eqnKloosterman1}}
By Lemma~\ref{lem:eqnKloosterman} we may assume 
\begin{align*}
V\gg q^{1/4r(r-1)}.
\end{align*}
We partition the interval $|v|\le V$ into disjoint intervals
\begin{align*}
[-V,V]=\bigcup_{j=1}^{K}I_j, \quad K\ll V/q^{1/4r(r-1)}, \quad |I_j|\ll q^{1/4r(r-1)},
\end{align*}
and let $K(I_{j_1},\dots,I_{j_{2r}})$ count the number of solutions to the congruence 
\begin{align*}
\frac{1}{\lambda+v_1}+\dots+\frac{1}{\lambda+v_r}\equiv \frac{1}{\lambda+v_{r+1}}+\dots+\frac{1}{\lambda+v_{2r}} \mod{q},
\end{align*}
with variables satisfying $v_i\in I_{j_i}$. By the pigeonhole principle, there exists some tuple $(j_1,\dots,j_{2r})$ such that 
\begin{align}
\label{eq:PGH}
K_r(\lambda,V)\ll \frac{V^{2r}}{q^{1/2(r-1)}}K(I_{j_1},\dots,I_{j_{2r}}).
\end{align}
Detecting via additive characters and applying H\"{o}lder's inequality, we have 
\begin{align*}
K(I_{j_1},\dots,I_{j_{2r}})&\le \frac{1}{q}\sum_{y=1}^{q}\prod_{i=1}^{2r}\left|\sum_{v\in I_{j_i}}e_q(y(\lambda+v)^{-1}) \right| \\ &\le \prod_{i=1}^{2r}\left(\frac{1}{q}\sum_{y=1}^{q}\left|\sum_{v\in I_{j_i}}e_q(y(\lambda+v)^{-1}) \right|^{2r} \right)^{1/2r},
\end{align*}
and hence by Lemma~\ref{lem:eqnKloosterman}
\begin{align*}
K(I_{j_1},\dots,I_{j_{2r}})\ll q^{1/4(r-1)+o(1)}.
\end{align*}
Combining with~\eqref{eq:PGH} we get 
\begin{align*}
K_{r}(\lambda,V)\ll \frac{V^{2r}q^{o(1)}}{q^{1/4(r-1)}},
\end{align*}
and completes the proof.
\section{Proof of Theorem~\ref{thm:main1}}
Assuming $q$ has factorization
\begin{align*}
q=q_1\prod_{k\in \cK}p_k^2\prod_{i\in \cI}p_i^{2\alpha_i}\prod_{j\in \cJ}p_j^{2\beta_j+1},
\end{align*}
for some sets of disjoint integers $\cK,\cI,\cJ,$ integers $\alpha_j\ge 2,\beta_j\ge 1$ and $q_1$ squarefree, we have 
\begin{align}
\label{eq:q0cfactorize}
s=\prod_{k\in \cK}p_k^2, \quad  c=\prod_{i\in \cI}p_i^{2\alpha_i}\prod_{j\in \cJ}p_j^{2\beta_j+1}.
\end{align}
With notation as in Lemma~\ref{lem:charKloos} 
\begin{align}
\label{eq:MVS76}
&\sum_{\lambda=1}^{q}\left|\sum_{1\le v \le V}\chi(\lambda+v) \right|^{2r}\ll qV^{r}+q^{1/2+o(1)}q_3q_4VS,
\end{align}
where 
\begin{align}
\label{eq:SSddd}
\nonumber S&=\sum_{\substack{d|q_5 \\\ t_2\dots t_{2r}|q_1 \\ s_2\dots s_{2r}|q_2 \\ t_js_j\ll V}}(t_1\dots t_{2r})^{1/2}d^{1/2}s_2\dots s_{2r}\prod_{j=2}^{2r}\max_{\lambda}K_{r,q_3q_4d}(\lambda,V/t_js_j)^{1/2r} \\ 
&=\sum_{\substack{d|q_5 \\\ t_2\dots t_{2r}|q_1 \\ s_2\dots s_{2r}|q_2 \\ t_js_j \ll V}}S(d,t_2,s_2,\dots,t_{2r},s_{2r}),
\end{align}
and 
\begin{align*}
&S(d,t_2,s_2,\dots,t_{2r},s_{2r})= \\ &  \quad \quad \quad \quad (t_2\dots t_{2r})^{1/2}d^{1/2}s_2\dots s_{2r}\prod_{j=2}^{2r}\max_{\lambda}K_{r,q_3q_4d}(\lambda,V/t_js_j)^{1/2r}.
\end{align*}
We recall that $q_2,\dots,q_5$ are given by
\begin{align*}
q_2=\prod_{k\in \cK}p_k, \quad q_3=\prod_{i\in \cI}p_i^{\alpha_i}, \quad q_4=\prod_{j\in \cJ}p_j^{\beta_j}, \quad q_5=\prod_{j\in \cJ}p_j.
\end{align*}
Fix some $d,t_2,\dots,t_{2r},s_2,\dots,s_{2r}$ satisfying 
\begin{align*}
d|q_5, \quad t_2\dots t_{2r}|q_1, \quad s_2\dots s_{2r}|q_2, \quad t_js_j\ll V,
\end{align*}
and consider $S(d,t_2,s_2,\dots,t_{2r},s_{2r})$. We partition the indicies $\{2,\dots,2r\}$ into sets
\begin{align*}
\cS_1&=\{ 2\le j \le 2r \ : \ t_js_j<V/(q_3q_4d)^{1/4r(r-1)}\}, \\ 
\cS_2&=\{ 2\le j \le 2r \ : \ V/(q_3q_4d)^{1/4r(r-1)}\le  t_js_j\ll V  \},
\end{align*}
and write 
\begin{align}
|\cS_1|=k_1, \quad |\cS_2|=k_2, \quad k_1+k_2=2r-1.
\end{align}
By Lemma~\ref{lem:eqnKloosterman1}, for any $0\le \lambda <q_3q_4d$ we have 
\begin{align}
\label{eq:Kb13}
K_{r,q_3q_4d}(\lambda,V/t_js_j)\ll \begin{cases}
q^{o(1)}(V/t_js_j)^{2r}\frac{1}{(q_3q_4d)^{1/4(r-1)}}, \quad j\in \cS_1, \\
q^{o(1)}(V/t_js_j)^{r}, \quad j\in \cS_2,
\end{cases}
\end{align}
which implies that
\begin{align*}
S(d,t_2,s_2,\dots,t_{2r},s_{2r})&\ll q^{o(1)}d^{1/2}\prod_{j\in \cS_1}\frac{t_j^{1/2}s_j}{(q_3q_4d)^{1/8r(r-1)}}\left(\frac{V}{t_js_j}\right)\prod_{j\in \cS_2}t_j^{1/2}s_j\left(\frac{V}{t_js_j}\right)^{1/2} \\
&\ll q^{o(1)}\frac{d^{1/2}}{(q_3q_4d)^{k_1/8r(r-1)}}\prod_{j\in \cS_2}s_j^{1/2}\prod_{j\in \cS_1}\frac{1}{t_j^{1/2}}V^{k_1+k_2/2} \\
&\ll q^{o(1)}q_5^{1/2}\left(\frac{V^{k_1/2}}{(q_3q_4q_5)^{k_1/8r(r-1)}}\prod_{j\in \cS_2}s_j^{1/2}\right)V^{r-1/2},
\end{align*}
using that $d|q_5$. Since each $s_j\ll V$ and $s_2\dots s_{2r}|q_2,$ we have 
\begin{align*}
\prod_{j\in \cS_2}s_j^{1/2}\ll \min \left\{V^{k_2},q_2 \right\},
\end{align*}
and hence 
\begin{align}
\label{eq:Sdb1}
S(d,t_2,s_2,\dots,t_{2r},s_{2r}) &\ll q^{o(1)}\frac{q_5^{1/2}V^{2r-1}}{(q_3q_4q_5)^{k_1/8r(r-1)}},
\end{align}
and 
\begin{align}
\label{eq:Sdb2}
S(d,t_2,s_2,\dots,t_{2r},s_{2r})\ll q^{o(1)}(q_5q_2)^{1/2}\left(\frac{V^{k_1/2}}{(q_3q_4q_5)^{k_1/8r(r-1)}}\right)V^{r-1/2}.
\end{align}
If $k_1=0$ then we use~\eqref{eq:Sdb2}, while if $k_1>0$ then we use~\eqref{eq:Sdb1}. This gives
\begin{align*}
S(d,t_2,s_2,\dots,t_{2r},s_{2r})\ll q^{o(1)}q_5^{1/2}\left(q_2^{1/2}V^{r-1/2}+\frac{V^{2r-1}}{(q_3q_4q_5)^{1/8r(r-1)}} \right),
\end{align*}
and hence from~\eqref{eq:MVS76},~\eqref{eq:SSddd}  and the estimate $d(n)=n^{o(1)}$ we get 
\begin{align*}
&\sum_{\lambda=1}^{q}\left|\sum_{1\le v \le V}\chi(\lambda+v) \right|^{2r}\ll qV^{r}
\\&+q^{1/2+o(1)}q_2^{1/2}q_3q_4q_5^{1/2}V^{r+1/2}+q^{1/2+o(1)}(q_3q_4q_5^{1/2})^{1-1/8r(r-1)}V^{2r},
\end{align*}
and the result follows since 
\begin{align*}
q_3q_4q_5^{1/2}=c^{1/2}, \quad q_2=s^{1/4}.
\end{align*}

\end{document}